\documentclass[11pt]{article}
\title{Completely almost periodic functionals}
\author{\textit{Volker Runde}}
\date{}
\renewcommand{\baselinestretch}{1.2}
\newcommand{\dated}{\mbox{} \hfill {\small [{\tt \today}]}}
\usepackage{amsmath,amssymb,amsfonts,diagrams}
\newenvironment{keywords}{\noindent\small {\it Keywords\/}:}{\vskip 4pt}
\newenvironment{classification}{\noindent\small 2000 {\it Mathematics Subject
Classification\/}:}{\vskip 12pt}

\newcommand{\tensor}{\otimes}

\newcommand{\Tensor}{\hat{\otimes}}

\newcommand{\wTensor}{\check{\otimes}}

\newcommand{\cstar}{{C^\ast}}

\newcommand{\id}{{\mathrm{id}}}
\newcommand{\cb}{{\mathrm{cb}}}

\newcommand{\A}{{\mathfrak A}}

\newcommand{\CB}{{\cal CB}}

\newcommand{\VN}{\operatorname{VN}}

\usepackage{amsthm,enumerate}
\theoremstyle{plain}
\newtheorem{theorem}{Theorem}[section]

\newtheorem{corollary}[theorem]{Corollary}
\newtheorem{proposition}[theorem]{Proposition}
\theoremstyle{definition}
\newtheorem{definition}[theorem]{Definition}
\theoremstyle{remark}

\newtheorem*{rems}{Remarks}
\newtheorem*{exs}{Examples}
\newenvironment{remarks}{\begin{rems}\begin{enumerate}}{\end{enumerate}\end{rems}}
\newenvironment{examples}{\begin{exs}\begin{enumerate}}{\end{enumerate}\end{exs}}
\newenvironment{items}{\begin{enumerate}[\rm (i)]}{\end{enumerate}}

\newcommand{\CK}{\mathcal{CK}}
\newcommand{\CAP}{\mathcal{CAP}}
\begin{document}
\maketitle
\begin{abstract}
Using the notion of complete compactness introduced by H.\ Saar, we define completely almost periodic functionals on completely contractive Banach algebras. We show that, if $(M,\Gamma)$ is a Hopf--von Neumann algebra with $M$ injective, then the space of completely almost periodic functionals on $M_\ast$ is a $\cstar$-subalgebra of $M$.
\end{abstract}
\begin{keywords}
completely compact map; completely almost periodic functional; Hopf--von Neumann algebra.
\end{keywords}
\begin{classification}
Primary 47L25; Secondary 22D25, 43A30, 46L07, 47L50.
\end{classification}
\section*{Introduction}
The almost periodic and weakly almost periodic continuous functions on a locally compact group $G$ form $\cstar$-subalgebras of $L^\infty(G)$, usually denoted by $\mathcal{AP}(G)$ and $\mathcal{WAP}(G)$, respectively: this is fairly elementary to prove and well known (see \cite{Bur} for instance). In a more abstract setting, one can define, for a general Banach algebra $\A$ the spaces $\mathcal{AP}(\A)$ and $\mathcal{WAP}(\A)$ of almost periodic and weakly almost periodic functionals on $\A$; if $\A = L^1(G)$, we have $\mathcal{AP}(\A) = \mathcal{AP}(G)$ and $\mathcal{WAP}(\A) = \mathcal{WAP}(G)$.
\par 
For $\A = A(G)$, Eymard's Fourier algebra (\cite{Eym}), the spaces $\mathcal{AP}(\A)$ and $\mathcal{WAP}(\A)$ are usually denoted by $\mathcal{AP}(\hat{G})$ and $\mathcal{WAP}(\hat{G})$: they were first considered by C.\ F.\ Dunkl and D.\ E.\ Ramirez (\cite{DR}) and further studied by E.\ E.\ Granirer (\cite{Gra} and \cite{Gra2}), A.\ T.-M.\ Lau (\cite{Lau}), and others. Except in a few special cases, e.g., if $G$ is abelian or discrete and amenable, it is unknown whether $\mathcal{AP}(\hat{G})$ and $\mathcal{WAP}(\hat{G})$ are $\cstar$-subalgebras of the group von Neumann algebra $\VN(G)$.
\par
Recently, M.\ Daws considered the almost periodic and weakly almost periodic functionals on the predual of a Hopf--von Neumann algebra with underlying von Neumann algebra $M$ (\cite{Daw}). He proved: \emph{If $M$ is abelian}, then both $\mathcal{AP}(M_\ast)$ and $\mathcal{WAP}(M_\ast)$ are $\cstar$-subalgebras of $M$ (\cite[Theorems 1 and 4]{Daw}). Unfortunately, the demand that $M$ be abelian is crucial for Daws' proofs to work (see \cite{Run} for a discussion).
\par
Over the past two decades, it has become apparent that purely Banach algebraic notions aren't well suited for the study of $A(G)$: one often has to tweak these notions in a way that takes the canonical operator space structure of $A(G)$---as the predual of the group von Neumann algebra---into account. For instance, Banach algebraic amenability of $A(G)$ forces $G$ to be finite-by-abelian (\cite{FR}) whereas $A(G)$ is \emph{operator} amenable if and only if $G$ is amenable (\cite{Rua}), a much more satisfactory result.
\par 
We apply this philosophy to almost periodicity. A functional $\phi$ on a Banach algebra $\A$ is called almost periodic if the maps
\begin{equation} \tag{\mbox{$\ast$}} \label{period}
  \A \to \A^\ast, \quad a \mapsto \left\{ \begin{array}{c} a \cdot \phi \\ \phi \cdot a \end{array} \right.
\end{equation}
are compact. Suppose now that $\A$ is a completely contractive Banach algebra. Then the maps (\ref{period}) are completely bounded. There are various definitions that attempt to fit the notion of a compact operator to a completely bounded context (see \cite{Saa}, \cite{Web}, \cite{Oik}, and \cite{Yew}, for instance). We focus on the definition of a completely compact map from \cite{Saa}, and define $\phi \in \A^\ast$ to be completely almost periodic if the maps (\ref{period}) are completely compact.
\par 
Our main result is that, if $(M,\Gamma)$ is a Hopf--von Neumann algebra such that $M$ is injective, then the completely almost periodic functionals on $M_\ast$ form a $\cstar$-subalgebra of $M$. This applies, in particular, to $A(G)$ in the cases where $G$ is amenable or connected. 
\section{Completely compact maps}
There are various ways to adapt the notion of a compact operator to an operator space setting: operator compactness (\cite{Web} and \cite{Yew}), complete compactness (\cite{Saa}), and Gelfand complete compactness (\cite{Oik}).
\par
The notion of a completely compact map between $\cstar$-algebras was introduced by H.\ Saar in his Diplomarbeit \cite{Saa} under G.\ Wittstock's
supervision. The starting point of his definition is the following observation: if $E$ and $F$ are Banach spaces, and $T \!: E \to F$ is a bounded linear map, then $T$ is compact if and only if, for each $\epsilon > 0$, there is a finite-dimensional subspace $Y_\epsilon$ of $F$ such that $\| Q_{Y_\epsilon} T \| < \epsilon$, where $Q_{Y_\epsilon} \!: F \to F / Y_\epsilon$ is the quotient map. 
\par 
This can be used to define an operator space analog of compactness, namely complete compactness. Saar didn't define complete compactness for maps between general, abstract operator spaces---simply because these objects hadn't been defined yet at that time---, but his definition obviously extends to general operator spaces. (Our reference for operator space theory is \cite{ER}, the notation and terminology of which we adopt.) 
In modern language, Saar's definition reads:
\begin{definition} \label{CKdef}
Let $E$ and $F$ be operator spaces. Then $T \in \CB(E,F)$ is called \emph{completely compact} if, for each $\epsilon > 0$, there is a finite-dimensional subspace $Y_\epsilon$ of $F$ such that $\| Q_{Y_\epsilon} T \|_\cb < \epsilon$, where $Q_{Y_\epsilon} \!: F \to F / Y_\epsilon$ is the quotient map. 
\end{definition}
\begin{remarks}
\item Trivially, completely compact maps are compact.
\item It is obvious that, if $E$ is a Banach space and $F$ is an operator space, then $T \in {\cal B}(E,F) = \CB(\max E, F)$ is completely compact if and only if it is compact.
\item On \cite[pp.\ 32--34]{Saa}, Saar constructs an example of a compact, completely bounded map on ${\cal K}(\ell^2)$ that fails to be completely compact.
\item Complete compactness may not be stable under co-restrictions, i.e., if $T \in \CB(E,F)$ be completely compact, and let $Y$ be a closed subspace of $F$ containing $TE$, then it is not clear why $T$ viewed as an element of $\CB(E,Y)$ should be completely compact.
\item Very recently, complete compactness was put to use for the study of so-called operator multipliers (see \cite{JLTT} and \cite{TT}). 
\end{remarks}
\par 
Given two operator spaces $E$ and $F$, we write $\CK(E,F)$ for the completely compact operators in $\CB(E,F)$.
\par 
The following proposition is essentially \cite[Lemma 1 a) and Lemma 2]{Saa}. (Of course, Saar only considers maps between $\cstar$-algebras, but his proofs carry over more or less verbatim.)
\begin{proposition} \label{Saarprop}
Let $E$ and $F$ be operator spaces. Then:
\begin{items}
\item $\CK(E,F)$ is a closed subspace of $\CB(E,F)$ containing all finite rank operators;
\item if $T \in \CK(E,F)$, $X$ is another operator space, and $R \in \CB(X,E)$, then $TR \in \CK(X,F)$;
\item if $T \in \CK(E,F)$, $Y$ is another operator space, and $S \in \CB(F,Y)$, then $ST \in \CK(E,Y)$.
\end{items}
\end{proposition}
\par 
From Schauder's theorem and Saar's characterization of compact operators, it follows immediately that a bounded linear operator $T$ from a Banach space $E$ into a Banach space $F$ is compact if and only if, for each $\epsilon > 0$, there is a closed subspace $X_\epsilon$ of $E$ with finite co-dimension such that $\| T |_{X_\epsilon} \| < \epsilon$ (compare also \cite{Lac}).
\par
Following \cite{Oik}, we define:
\begin{definition} \label{GCKdef}
Let $E$ and $F$ be operator spaces. Then $T \in \CB(E,F)$ is called \emph{Gelfand completely compact} if, for each $\epsilon > 0$, there is a closed subspace $X_\epsilon$ of $E$ with finite co-dimension such that $\| T |_{X_\epsilon} \|_\cb < \epsilon$. 
\end{definition}
\begin{remarks}
\item Obviously, $T \in \CB(E,F)$ is Gelfand completely compact if and only if $T^\ast \in \CB(F^\ast,E^\ast)$ is completely compact (and vice versa).
\item A result analogous to Proposition \ref{Saarprop} holds for Gelfand completely compact maps.
\item Based on results from the unpublished paper \cite{Web}, T.\ Oikhberg gives examples of completely compact maps that fail to be Gelfand completely compact and vice versa (\cite[pp.\ 155-156]{Oik}). Hence, an analog for Schauder's theorem fails for complete compactness.
\end{remarks}
\par
Under certain circumstances, every completely compact map is Gelfand completely compact (\cite[Theorem 3.1]{Oik}). For our purposes, the following is important (see \cite[p.\ 70]{ER} for the notion of an injective operator space):
\begin{proposition} \label{prop2}
Let $E$ and $F$ be operator spaces such that $E^\ast$ and $F^\ast$ are injective. Then the following are equivalent for $T \in \CB(E,F^\ast)$:
\begin{items}
\item $T$ is completely compact;
\item $T$ is Gelfand completely compact;
\item $T$ is a $\cb$-norm limit of finite rank operators.
\end{items}
\end{proposition}
\begin{proof}
Obviously, (iii) implies both (i) and (ii).
\par
(ii) $\Longrightarrow$ (iii): Let $T$ be Gelfand completely compact, and let $\epsilon > 0$. Then there is a closed subspace $X_\epsilon$ of $E$ with finite co-dimension such that $\| T |_{X_\epsilon} \|_\cb < \epsilon$. Due to the injectivity of $F^\ast$, there is $\tilde{T} \in \CB(E,F^\ast)$ such that $\tilde{T} |_{X_\epsilon} = T |_{X_\epsilon}$ and $\| \tilde{T} \|_\cb = \| T |_{X_\epsilon} \|_\cb < \epsilon$. Then $S := T - \tilde{T}$ satisfies $\| S - T \|_\cb < \epsilon$ and vanishes on $X_\epsilon$; since $X_\epsilon$ has finite co-dimension, $S$ must have finite rank..
\par 
(i) $\Longrightarrow$ (iii): As $T^\ast \in \CB(F^{\ast\ast},E^\ast)$ is Gelfand completely compact, it is a $\cb$-norm limit of finite rank operators---by the argument used for (ii) $\Longrightarrow$ (iii)---as is its adjoint $T^{\ast\ast} \in \CB(E^{\ast\ast}, F^{\ast\ast\ast})$. Hence, given $\epsilon > 0$, there is a finite rank operator $S \!: E^{\ast\ast} \to F^{\ast\ast\ast}$ such that $\| S - T^{\ast\ast} \|_\cb < \epsilon$.
Let $Q \!: F^{\ast\ast\ast} \to F^\ast$ be the Dixmier projection, i.e., the adjoint of the canonical embedding of $F$ into $F^{\ast\ast}$. Then $S_0 := QS |_E$ is a finite rank operator from $E$ to $F^\ast$ such that $\| S_0 - T \|_\cb < \epsilon$.
\end{proof}
\section{Completely almost periodic functionals}
If $\A$ is a Banach algebra, then its dual spaces becomes a Banach $\A$-bimodule via
\[
  \langle x, a \cdot \phi \rangle := \langle xa, \phi \rangle \quad\text{and}\quad \langle x , \phi \cdot a \rangle := \langle ax, \phi \rangle
  \qquad (a,x \in \A, \, \phi \in \A^\ast).
\]
For $\phi \in \A^\ast$, define $L_\phi, R_\phi \!: \A \to \A^\ast$ via
\[
  L_\phi a := \phi \cdot a \quad\text{and}\quad R_\phi a := a \cdot \phi \qquad (a \in \A).
\]
A functional $\phi \in \A^\ast$ is commonly called \emph{almost periodic} if $L_\phi$ and $R_\phi$ are compact operators. (As $L_\phi = R_\phi^\ast |_\A$ and $R_\phi = L_\phi^\ast |_\A$, it is sufficient to require that only one of $L_\phi$ and $R_\phi$ be compact.) We denote the space of all almost periodic functionals in $\A^\ast$ by $\mathcal{AP}(\A)$.
\par 
An operator space that is also an algebra such that multiplication is completely contractive, is called a \emph{completely contractive Banach algebra}.
\par 
We define:
\begin{definition} \label{capdef}
Let $\A$ be a completely contractive Banach algebra. We call $\phi \in \A^\ast$ \emph{completely almost periodic} if $L_\phi , R_\phi \in \CK(\A,\A^\ast)$ and denote the collection of completely almost periodic functionals in $\A^\ast$ by $\CAP(\A)$.
\end{definition}
\begin{remarks}
\item Obviously, $\CAP(\A)$ is a closed linear subspace of $\A^\ast$.
\item Since Schauder's theorem is no longer true for complete compactness, we do need the requirement that both $L_\phi$ and $R_\phi$ are completely compact.
\item Since $L_\phi = R_\phi^\ast |_\A$ and $R_\phi = L_\phi^\ast |_\A$, we could have replaced in complete almost periodicity---based on complete compactness---in Definition \ref{capdef} by the analogous notion involving Gelfand complete compactness instead: we would still have obtained the same functionals.
\end{remarks}
\par
We shall now look at a special class of completely contractive Banach algebras which arise naturally in abstract harmonic analysis.
\par 
Recall that a \emph{Hopf--von Neumann algebra} is a pair $(M,\Gamma)$ where $M$ is a von Neumann algebra and $\Gamma \!: M \to M \bar{\tensor} M$ is a \emph{co-multiplication}, i.e., a normal, unital, $^\ast$-homomorphism satisfying
\[
  (\id \tensor \Gamma) \circ \Gamma = (\Gamma \tensor \id) \circ \Gamma.
\]
The co-multiplication induces a product $\ast$ on $M_\ast$ via
\[
  \langle f \ast g, x \rangle := \langle f \tensor g, \Gamma x \rangle \qquad (f,g \in M_\ast, \, x \in M),
\]
which turns $M_\ast$ into a completely contractive Banach algebra.
\begin{examples}
\item Let $G$ be a locally compact group, and define a co-multiplication $\Gamma \!: L^\infty(G) \to L^\infty(G) \bar{\tensor} L^\infty(G)$---noting that $L^\infty(G) \bar{\tensor} L^\infty(G) \cong L^\infty(G \times G)$---by letting
\[
  (\Gamma \phi)(x,y) := \phi (xy) \qquad (x,y \in G, \, \phi \in L^\infty(G)).
\]
Then $L^\infty(G)_\ast = L^1(G)$ with the product induced by $\Gamma$ is just the usual convolution algebra $L^1(G)$. Since the operator space structure on $L^1(G)$ is maximal, $\CAP(L^1(G))$ equals $\mathcal{AP}(L^1(G))$; it consists precisely of the almost periodic continuous functions on $G$ in the classical sense (see \cite{Bur}, for instance) and thus, in particular, is a $\cstar$-subalgebra of $L^\infty(G)$.
\item Let again $G$ be a locally compact group, let $\VN(G)$ be its group von Neumann algebra, i.e., the von Neumann algebra generated by $\lambda(G)$, where $\lambda$ is the left regular representation of $G$ on $L^2(G)$, and let $\Gamma \!: \VN(G) \to \VN(G) \bar{\tensor} \VN(G)$ be given by 
\[
  \Gamma \lambda(x) = \lambda(x) \tensor \lambda(x) \qquad (x \in G).
\]
Then $\VN(G)_\ast$ is Eymard's Fourier algebra $A(G)$ (\cite{Eym}), and the product induced by $\Gamma$ is pointwise multiplication. The space $\mathcal{AP}(A(G))$---often denoted by $\mathcal{AP}(\hat{G})$---was first considered in \cite{DR} and further studied in \cite{Gra} and \cite{Lau}, for instance.
\end{examples}
\par 
Except in a few special cases, e.g., if $G$ is abelian or discrete and amenable (as a consequence of \cite[Theorem 12]{Gra} and \cite[Proposition 2]{Gra2}), is unknown whether or not $\mathcal{AP}(\hat{G})$ is a $\cstar$-subalgebra of $\VN(G)$.
\par
The picture changes once we replace almost periodicity with complete almost periodicity:
\begin{theorem} \label{mainthm}
Let $(M,\Gamma)$ be a Hopf--von Neumann algebra such that $M$ is injective. Then $\CAP(M_\ast)$ is a $\cstar$-subalgebra of $M$.
\end{theorem}
\begin{proof}
Let $x \in M$. Then, by Proposition \ref{prop2}, $L_x$ is completely compact if and only if it is a $\cb$-norm limit of finite rank operators.
In view of the completely isometric identifications (\cite[Corollary 7.1.5]{ER} and \cite[Theorem 7.2.4]{ER})
\[
  M \bar{\tensor} M \cong (M_\ast \Tensor M_\ast)^\ast \cong \CB(M_\ast,M)
\]
and of \cite[Proposition 8.1.2]{ER}, this means that $L_x$ is completely compact if and only if $\Gamma x \in M \wTensor M$. A similar assertion holds for $R_x$.
\par 
All in all, we have that
\[
  \CAP(M_\ast) = \{ x \in M : \Gamma x \in M \wTensor M \}.
\]
Since $M \wTensor M$ is nothing but the spatial $\cstar$-tensor product of $M$ with itselt and thus $\cstar$-subalgebra of $M \bar{\tensor} M$, and since $\Gamma$ is a $^\ast$-homomorphism, this proves the claim.
\end{proof}
\par 
If $G$ is an amenable or connected locally compact group, then $\VN(G)$ is well known to be injective. Writing $\CAP(\hat{G})$ for $\CAP(A(G))$, we thus have:
\begin{corollary}
Let $G$ be an amenable or connected locally compact group. Then $\CAP(\hat{G})$ is a $\cstar$-subalgebra of $\VN(G)$.
\end{corollary}
\begin{remarks}
\item If $G$ is abelian---or, more generally, finite-by-abelian---, then the canonical operator space structure on $A(G)$ is equivalent to $\max A(G)$, so that $\CAP(\hat{G}) = \mathcal{AP}(\hat{G})$. 
\item Suppose that $G$ is discrete and amenable. Then $\mathcal{AP}(\hat{G})$ equals $C^\ast_r(G)$, the (reduced) group $\cstar$-algebra of $G$ by \cite[Proposition 3(b)]{Gra2}. Since $\Gamma C^\ast_r(G) \subset C^\ast_r(G) \wTensor C^\ast_r(G) \subset \VN(G) \wTensor \VN(G)$, an inspection of the proof of Theorem \ref{mainthm} shows that therefore $\mathcal{AP}(\hat{G}) \subset \CAP(\hat{G})$ and thus, trivially, $\mathcal{AP}(\hat{G}) = \CAP(\hat{G})$.
\end{remarks}
\renewcommand{\baselinestretch}{1.0}
\renewcommand{\baselinestretch}{1.2}
\dated
\vfill
\begin{tabbing}
\textit{Author's address}: \= Department of Mathematical and Statistical Sciences \\
\> University of Alberta \\
\> Edmonton, Alberta \\
\> Canada T6G 2G1 \\[\medskipamount]
\textit{E-mail}: \> \texttt{vrunde@ualberta.ca} \\[\medskipamount]
\textit{URL}: \> \texttt{http://www.math.ualberta.ca/$^\sim$runde/}
\end{tabbing}           

\begin{thebibliography}{J--L--T--T} \begin{small}
%
\bibitem[Bur]{Bur} \textsc{R.\ B.\ Burckel}, \textit{Weakly Almost Periodic Functions on Semigroups}. Gordon and Breach, 1970.
%
\bibitem[Daw]{Daw} \textsc{M.\ Daws}, Weakly almost periodic functionals on the measure algebra. \textit{Math.\ Z.}\ \textbf{265} (2010), 285-–296. .
%
\bibitem[D--R]{DR} \textsc{C.\ F.\ Dunkl} and \textsc{D.\ E.\ Ramirez}, Weakly almost periodic functionals on the Fourier algebra. \textit{Trans.\ Amer.\ Math.\ Soc.}\ \textbf{185} (1973), 501-–514.
%
\bibitem[E--R]{ER} \textsc{E.\ G.\ Effros} and \textsc{Z.-J.\ Ruan}, \textit{Operator Spaces}. London Mathematical Society Monographs (New Series) \textbf{25}, Clarendon Press, 2000.
%
\bibitem[Eym]{Eym} \textsc{P.\ Eymard}, L'alg\`ebre de Fourier d'un groupe localement compact. \textit{Bull.\ Soc.\ Math.\ France}
\textbf{92} (1964), 181--236. 
%
\bibitem[F--R]{FR} \textsc{B.\ E.\ Forrest} and \textsc{V.\ Runde}, Amenability and weak amenability of the Fourier algebra. \textit{Math.\ Z.}\ \textbf{250} (2005), 731--744. 
%
\bibitem[Gra 1]{Gra} \textsc{E.\ E.\ Granirer}, Weakly almost periodic and uniformly continuous functionals on the Fourier algebr of any locally compact group. \textit{Trans.\ Amer.\ Math.\ Soc.}\ \textbf{189} (1974), 371--382.
%
\bibitem[Gra 2]{Gra2} \textsc{E.\ E.\ Granirer}, Density theorems for some linear subspaces and some $C^\ast$-subalgebras of $\VN(G)$.  In: Symposia Mathematica \textbf{XXII}, pp.\ 61--70, Academic Press, London--New York, 1977. 
%
\bibitem[J--L--T--T]{JLTT} \textsc{K.\ Juschenko}, \textsc{R.\ H.\ Levene}, \textsc{I.\ G.\ Todorov}, and \textsc{L.\ Turowska}, Compactness properties of operator multipliers. \textit{J.\ Funct.\ Anal.}\ \textbf{256} (2009), 3772–-3805. 
%
\bibitem[Lac]{Lac} \textsc{H.\ E.\ Lacey}, \textit{Generalizations of Compact Operators in Locally Convex Topological Linear Spaces}. Ph.D.\ Thesis, New Mexico State University, 1963.
%
\bibitem[Lau]{Lau} \textsc{A.\ T.-M.\ Lau}, Uniformly continuous functionals on the Fourier algebr of any locally compact group. \textit{Trans.\ Amer.\ Math.\ Soc.}\ \textbf{251} (1974), 371--382.
%
\bibitem[Oik]{Oik} \textsc{T.\ Oikhberg}, Direct sums of operator spaces. \textit{J.\ London Math.\ Soc.}\ (2) \textbf{64} (2001), 144--160.
%
\bibitem[Pal]{Pal} \textsc{T.\ W.\ Palmer}, \textit{Banach Algebras and the General Theory of $^\ast$-Algebras. Volume I: Algebras and Banach Algebras}. Ecyclopedia of Mathematics and its Applications \textbf{49}, Cambridge University Press, 1994.
%
\bibitem[Rua]{Rua} \textsc{Z.-J.\ Ruan}, The operator amenability of $A(G)$. \textit{Amer.\ J.\ Math.}\ \textbf{117} (1995), 1449--1474.
%
\bibitem[Run]{Run} \textsc{V.\ Runde}, Co-representations of Hopf--von Neumann algebras on operator spaces other than column Hilbert space. \textit{Bull.\ Aust.\ Math.\ Soc.}\ \textbf{82} (2010), 205--210.
%
\bibitem[Saa]{Saa} \textsc{H.\ Saar}, \textit{Kompakte, vollst\"andig beschr\"ankte Abbildungen mit Werten in einer nuklearen $\cstar$-Algebra}. Diplomarbeit, Universit\"at des Saarlandes, 1982.
%
\bibitem[T--T]{TT} \textsc{I.\ G.\ Todorov} and \textsc{L.\ Turowska}, Schur and operator multipliers. In: \text{R.\ J.\ Loy}, \textsc{V.\ Runde}, and \textsc{A.\ So{\l}tysiak} (eds.), \textit{Banach Algebras 2009}. Banach Center Publications, to appear.
%
\bibitem[Web]{Web} \textsc{C.\ Webster}, Matrix compact sets and operator approximation properties. ArXiv: \texttt{math/9894093}.
%
\bibitem[Yew]{Yew} \textsc{K.\ L.\ Yew}, Compactness and approximation property with respect to an operator space. \textit{Indiana Univ.\ Math.\ J.}\ \textbf{56} (2007), 3075--3128.
%
\end{small} \end{thebibliography}
\end{document}